\documentclass[12pt, english]{article}

\usepackage{babel}
\usepackage[latin1]{inputenc}
\usepackage{amsmath}
\usepackage{amssymb}
\usepackage{amsthm}

\theoremstyle{plain}
\newtheorem{Thm}{Theorem}[section]
\newtheorem{Prop}[Thm]{Proposition}

\theoremstyle{definition}

\newtheorem{Def*}{Definition}

\begin{document}

\title{Mean field matching and TSP in pseudo-dimension 1}

\author{ Giorgio Parisi \\
\small Dipartimento di Fisica\\[-0.8ex] \small
Universit\`a di Roma, La Sapienza, Rome, Italy.
\\[-0.8ex] \small
\texttt{giorgio.parisi@Roma1.infn.it
}\\
\\
Johan W\"astlund \\
\small  Department of Mathematical Sciences\\[-0.8ex] \small
Chalmers University of Technology, Gothenburg, Sweden\\[-0.8ex] \small
\texttt{wastlund@chalmers.se}}

\date{\small \today\\
%\small Mathematics Subject Classification: Primary: 60C05, 90C27, 90C35
}

\maketitle

\begin{abstract}
Recent work on optimization problems in random link models has verified several conjectures originating in statistical physics and the replica and cavity methods. In particular the numerical value 2.0415 for the limit length of a traveling salesman tour in a complete graph with uniform $[0,1]$ edge lengths has been established.

In this paper we show that the crucial integral equation obtained with the cavity method has a unique solution, and that the limit ground state energy obtained from this solution agrees with the rigorously derived value. Moreover, the method by which we establish uniqueness of the solution turns out to yield a new completely rigorous derivation of the limit.
\end{abstract}

\section{Introduction} \label{S:intro}
In \cite{W09}, the minimum matching and traveling salesman problems were studied in the pseudo-dimension $d$ mean field (or random link) model for $d\geq 1$. It was shown that certain predictions of \cite{KM89, MP85, MP86a, MP86b, MP87} based on the replica method are indeed correct. Here we show that the case $d=1$ allows stronger and more detailed conclusions, and we clarify the relation to the earlier results in \cite{W10tsp}.

The simplest random model corresponding to $d=1$ is the complete graph $K_n$ on $n$ vertices, with independent lengths from uniform distribution on the interval $[0,1]$ associated to the edges. We consider only this model, although the results, ultimately based on the local tree structure of the relatively short edges, remain valid in a number of similar models.

The minimum matching problem asks for a set of $n/2$ edges of minimum total length under the constraint that each vertex must be incident to exactly one edge. This requires $n$ to be even, but for odd $n$ we may allow one vertex to be left out of the pairing. It is known that the asymptotic behavior of the optimum solution remains the same even if we only require $n/2 - O(1)$ disjoint edges, in other words if we allow any fixed number of vertices to remain unmatched.

The traveling salesman problem (TSP) asks for a tour of minimum total length visiting every vertex exactly once. Since the triangle inequality need not hold, there will in general be shorter walks visiting each vertex and returning to the starting point if the same vertex can be visited several times. If such walks are permitted, one may or may not allow the same edge to be traversed more than once. Clearly there are several possible interpretations of the TSP, but we study the strictest one in which we ask for a cycle of $n$ edges.

The two problems were studied with the replica and cavity methods in \cite{KM89, MP85, MP86a, MP86b, MP87}, and among the results were predictions about the large $n$ limit of the total length of the solution, or in physical language the ground state energy in the thermodynamical limit. The idea is that as $n\to\infty$, the length $L_n$ of the optimal solution, which is a random quantity for each $n$, converges to a non-random limit $L^\star$. One may conjecture on fairly general grounds that $E(L_n) \to L^\star$ and that $L_n$ is ``self-averaging'' so that $L_n \to L^\star$ in probability. Remarkably, methods of physics allow for precise calculation of the limits $L^\star_M$ and $L^\star_{TSP}$ for matching and TSP respectively.

We can also define the $k$-factor problem where we ask for a set of $kn/2$ edges of minimum total length under the constraint that each vertex must be incident to exactly $k$ edges. Clearly $k=1$ is the matching problem, and the case $k=2$ is a relaxation of the TSP allowing multiple cycles. A nontrivial result, implicit in the early physics literature and rigorously proved by A.~Frieze \cite{F04}, tells us that in the large $n$ limit the length of the 2-factor and of the TSP are the same. In principle the results presented here can be generalized to the $k$-factor problem for generic $k$, but the computations become less explicit.

\subsection{The replica and cavity results}
We briefly recall some of the results of \cite{KM89, MP85}. Both problems lead to certain integral equations for the so-called \emph{order parameter function}. For the matching problem the equation is \begin{equation} \label{MMeq} G(x) = \int_{-x}^\infty e^{-G(y)}\, dy,\end{equation} and the ground state energy is given by \begin{equation} \label{MMground} L^\star_M = \frac12\int_{-\infty}^{+\infty} G(x)e^{-G(x)}\, dx.\end{equation}
For the TSP the equations take a similar form. The order parameter function $G$ has to satisfy \begin{equation} \label{TSPeq} G(x) = \int_{-x}^\infty (1+G(y))e^{-G(y)}\, dy,\end{equation} and the ground state energy is \begin{equation} \label{TSPground} L^\star_{TSP} = \frac12\int_{-\infty}^{+\infty} G(x)(1+G(x))e^{-G(x)}\, dx.\end{equation}
Here we consider only the case $r=0$ (in the notation of \cite{KM89, MP85}), corresponding to $d=1$ in \cite{W09}.

The equation \eqref{MMeq} corresponding to minimum matching has the explicit solution $G(x) = \log(1+e^x)$, and the ground state energy is $L^\star_M = \pi^2/12$. There does not seem to be an explicit solution to the analogous equation \eqref{TSPeq} for the TSP, but in \cite{KM89} a numerical solution led to $L^\star_{TSP} \approx 2.0415$, even though there was no proof that \eqref{TSPeq} has a solution or that such a solution must be unique.

\subsection{Rigorous results}
The $\pi^2/12$-limit for matching was established rigorously by David Aldous in 2001 \cite{A92, A01}. The method was related to the physics approach, and used the solution to \eqref{MMeq}. A similar approach to the TSP was indicated in \cite{A01}, but the main obstacle at the time seems to have been that \eqref{TSPeq} was not known to have a solution.

In \cite{W10tsp} the limit $L^\star_{TSP}$ of the TSP was determined with a different method. The result (conjectured in the technical report \cite{W05exact}) was
\begin{equation} \label{TSPrig} L^\star_{TSP} = \frac12\int_0^\infty y\,dx,\end{equation} where $y$ as a function of $x$ is defined by $y > 0$ and \begin{equation} \label{tspeq} \left(1+\frac x2\right)e^{-x} +  \left(1+\frac y2\right)e^{-y} = 1.\end{equation}
This led to the question whether the numbers given by \eqref{TSPground} and \eqref{TSPrig} are equal, and to the hope that a solution to \eqref{TSPeq} could somehow be reverse-engineered from \eqref{tspeq}.

\section{Agreement on the TSP} \label{S:cavity}
The first new result of this paper is a proof that equation \eqref{TSPeq} has a unique solution, and that the characterization of $L^\star_{TSP}$ by \eqref{TSPground} agrees with \eqref{TSPrig}.

\begin{Prop}
The integral equation \eqref{TSPeq} has a unique solution.
\end{Prop}

\begin{proof}
We introduce the auxiliary function $T$ given by $T(g) = (1+g)e^{-g}$. It follows from \eqref{TSPeq} that \begin{equation} \label{diff} \frac{d}{dx} G(x) = T(G(-x)),\end{equation} and similarly $$\frac{d}{dx} G(-x) = -T(G(x)).$$ Hence \begin{equation} \label{even} G'(x)T(G(x)) = G'(x)G'(-x) = G'(-x)T(G(-x)).\end{equation}
Now let $W$ be the primitive to $T$ for which $W(0) = 0$, or explicitly, $$W(g) = 2-2e^{-g}-ge^{-g}.$$ Then by \eqref{even}, $$\frac{d}{dx} W(G(x)) + \frac{d}{dx} W(G(-x)) = 0.$$ Hence $W(G(x)) + W(G(-x))$ is a constant, which has to be 2 by the boundary conditions. After simplification, the equation is
\begin{equation} \label{explicit} \left(2 + G(x)\right)e^{-G(x)} + \left(2 + G(-x)\right)e^{-G(-x)} = 2.\end{equation}
At this point the similarity to \eqref{tspeq} becomes apparent. If we let $\Lambda$ be the function that maps $x>0$ to the positive solution $y$ to \eqref{tspeq}, then \eqref{explicit} says that $G(-x) = \Lambda(G(x))$. In particular, $G(0)\approx 1.146$ is the unique positive solution to the equation $$(2+G(0))e^{-G(0)} = 1.$$
Replacing $G(-x)$ by $\Lambda(G(x))$ in \eqref{diff}, we obtain $$G'(x) = T(\Lambda(G(x))),$$ or equivalently $$\frac{G'(x)}{T(\Lambda(G(x)))} = 1.$$

Although not as explicit as one would first hope, we have arrived at a differential equation relating $G'(x)$ to $G(x)$ without involving $G(-x)$.
Integrating, we obtain \begin{equation} \label{integral} x = \int_{G(0)}^{G(x)} \frac{dx}{T(\Lambda(x))}.\end{equation} Since the integrand is positive and $G(0)$ is known, $G(x)$ is uniquely determined by \eqref{integral}. Conversely, it is clear that the function $G$ defined by \eqref{integral} is a solution to \eqref{TSPeq}.
\end{proof}

Remarkably, the ground state limit $L^\star_{TSP}$ can be found in terms of $\Lambda$ directly from \eqref{explicit}, without using the uniqueness of the solution.

\begin{Prop}
The two characterizations of $L^\star_{TSP}$ are consistent. In other words, the right hand side of \eqref{TSPground} is equal to the right hand side of \eqref{TSPrig}. \end{Prop}

\begin{proof}
In view of \eqref{diff}, equation \eqref{TSPground} can be written \begin{multline} \frac12\int_{-\infty}^\infty G(x)G'(-x)\,dx = \frac12\int_{-\infty}^\infty G'(x)G(-x)\,dx = \frac12\int_{-\infty}^\infty G'(x)\Lambda(G(x))\,du \\= \frac12\int_0^\infty \Lambda(t)\,dt,\end{multline} by the substitution $t = G(x)$.  This is the same thing as \eqref{TSPrig}.
\end{proof}

If instead we let $T(g) = e^{-g}$, we obtain in the same way the limit $L^\star_M$ for the matching problem. In that case the solution is explicit, with $W(g) = 1-e^{-g}$ and $\Lambda(t) = -\log(1-e^{-t})$.

\section{Rigorizing the replica results}
The proof that the results of \cite{W10tsp} are in agreement with the replica and cavity predictions is in itself satisfying as it shows that the inherently non-rigorous approach from statistical mechanics indeed gives a correct result.

Even more interesting is that the trick that transformed the integral equation \eqref{TSPeq} into an ordinary differential equation can produce an entirely rigorous proof of the TSP ground state limit independently of the results in \cite{W10tsp} (in view of the discussion of the TSP in \cite{A01} this is perhaps not that surprising). We first consider the technically simpler minimum matching problem, and later return to the TSP.

\subsection{Rescaling and diluted relaxation}
 It is convenient at this point to scale up the edge-lengths by a factor $n$ in order to obtain a local limit of the random model. We therefore let the edge-lengths be uniform on the interval $[0,n]$, which means that the total length of the minimum matching will be of order $n$.

We introduce another parameter $\lambda$ and study the \emph{diluted} relaxation of minimum matching. This relaxation consists in allowing any partial matching as a feasible solution, and letting the cost of a solution be the total length of the edges in the matching plus a penalty of $\lambda/2$ for each unmatched vertex.

Clearly edges of length greater than $\lambda$ cannot participate in the optimum solution, since it is less costly to leave the two endpoints unmatched and pay a penalty of $2\cdot \lambda/2 = \lambda$. Therefore the diluted relaxation is essentially a problem on an Erd\"os-R\'enyi random graph (sometimes called a Poisson Bethe lattice in the physics literature) where edges are present with probability $\lambda/n$.

It was shown in \cite{W09} (and in a different setting already in \cite{A92}) that in order to find the limit $L^\star_M$ of the minimum length of a perfect matching, it suffices to study the large $n$ limit of the diluted matching problem for fixed $\lambda$, and finally let $\lambda\to \infty$. Therefore we now leave the perfect matching problem and regard it only as a large $\lambda$ limit of the diluted problem.

\subsection{An exploration game}
A two-person perfect information zero-sum game called \emph{Exploration} was introduced in \cite{W09}. The two players Alice and Bob take turns choosing the edges of a self-avoiding walk starting from a preassigned vertex of a graph with lengths associated to the edges. At every move, the moving player pays an amount equal to the length of the chosen edge to the opponent. Before each move, the moving player also has the option of terminating the game and paying a penalty of $\lambda/2$ to the opponent. Each player is trying to maximize their total payoff (what they receive minus what they pay throughout the game).

As was shown in \cite{W09}, Exploration is connected to the diluted matching problem:
\begin{Prop} \label{P:firstMove} In a finite graph, Alice's optimal first move is to move along the edge incident to the starting point in the solution to the diluted matching problem if there is such an edge, and otherwise to pay $\lambda/2$ to Bob and terminate the game immediately. \end{Prop}

Hence in order to find the asymptotic total cost of the minimum diluted matching, we can study the probability distribution of the cost of Alice's first move in Exploration starting from an arbitrary vertex.

\subsection{Tree approximation}
The Poisson Weighted Infinite Tree (PWIT) was introduced by David Aldous \cite{A92, A01}. The PWIT is an infinite rooted tree where each vertex has a countably infinite sequence of children. The edges to the children have lengths given by a rate 1 Poissson point process on the positive real numbers (independent processes for all vertices).

The relevance of the PWIT in this context comes from the fact that it is a \emph{local limit} of $K_n$ (under the rescaled edge lengths). In \cite{A92, A01} a concept of \emph{weak} limit was used, but the relaxation to finite $\lambda$ allows us to work with a stronger and simpler form of local limit.

If $k$ is a positive integer, we let the $(k, \lambda)$-neighborhood of a vertex $v$ in a graph be the subgraph that can be reached by walking at most $k$ steps from $v$ along edges of length at most $\lambda$. We now compare the $(k,\lambda)$-neighborhood of the root of the PWIT with the $(k,\lambda)$-neighborhood of an arbitrarily chosen vertex $v$ of the complete graph $K_n$.

For fixed $k$ and $\lambda$, an event $E$ in $K_n$ is \emph{$(k, \lambda)$-invariant} if it depends only on the isomorphism type of the $(k,\lambda)$-neighborhood of $v$. For such an event $E$ we can ask for the probability $P_{PWIT}(E)$ of the corresponding event on the PWIT, with the root corresponding to $v$. We want to compare it to the probability $P_n(E)$ of $E$ on $K_n$.

\begin{Prop} \label{P:coupling} For fixed $k$, $\lambda$ and $E$, $P_n(E) \to P_{PWIT}(E)$ as $n\to\infty$.
\end{Prop}
This means that the large $n$ asymptotic probability of an event which depends only on the neighborhood of a particular vertex in $K_n$ can be found by instead studying the corresponding event on the PWIT.

\subsection{Exploration on the PWIT}
In view of Propositions \ref{P:firstMove} and \ref{P:coupling}, it makes sense to study Exploration played on the PWIT. The difficulty is that Proposition~\ref{P:coupling} concerns only the first $k$ levels of the PWIT, while there is no bound on the number of moves in Exploration. As was shown in \cite{W09}, this difficulty can be handled by introducing a concept of \emph{valuation}.

By the \emph{$\lambda$-cluster}, we mean the component of the root of the PWIT in the subgraph containing only edges of length at most $\lambda$. For the reader familiar with Galton-Watson processes, we remark that the underlying graph of the $\lambda$-cluster is a Poisson($\lambda$) Galton-Watson process.

A function $f$ from the vertices of the $\lambda$-cluster to the interval $[-\lambda/2, \lambda/2]$ is called a \emph{valuation} if for every $v$ in the $\lambda$-cluster it satisfies \begin{equation} \label{valdef} f(v) = \min(\lambda/2, l_i - f(v_i)),\end{equation} where the minimum is taken over $\lambda/2$ and the children $v_i$ of $v$, and $l_i$ is the length of the edge from $v$ to $v_i$. A valuation can be thought of as a consistent way for a player to assign a value to having moved to a particular vertex. The value should represent the total future payoff to the player who just moved to the vertex. Indeed, if the $\lambda$-cluster is finite, there is only one valuation, and it is given by the total future payoff under optimal play. Since the $\lambda$-cluster can be infinite, there are potentially several different valuations.

There is a simple way of constructing a valuation. For integer $k\geq 0$ we can construct a \emph{partial valuation} by assigning arbitrary values to the vertices at distance $k$ from the root, and then propagating these values towards the root according to \eqref{valdef}. We define partial valuations $f_A^{(k)}$ and $f_B^{(k)}$ by assigning values at distance $k$ in favor of Alice and Bob respectively. More precisely, if $k$ is even, $f^{(k)}_A(v) = -\theta/2$ and $f^{(k)}_B(v)=\theta/2$, while if $k$ is odd, $f^{(k)}_A(v) = \theta/2$ and $f^{(k)}_B(v)=-\theta/2$. As $k\to\infty$, $f^{(k)}_B$ converges pointwise to a (complete) valuation $f_B$. Similarly $f^{(k)}_A$ converges to a valuation $f_A$.

As is shown in \cite{W09}, the justification of the replica symmetric predictions for the matching problem reduces to showing that \begin{equation} \label{root} E\left[f_B^{(k)}(root) - f_A^{(k)}(root)\right] \to 0\end{equation} as $k\to\infty$.
In \cite{W09}, this is proved for the more general pseudo-dimension $d\geq 1$ case. The method is slightly non-constructive and consists in showing that there is only one (complete) valuation $f$. Since $f_A^{(k)}$ and $f_B^{(k)}$ both have to converge pointwise to $f$ as $k\to\infty$, \eqref{root} then follows from the principle of monotone convergence.
Here we show that the case $d=1$ allows a more direct proof of \eqref{root}.

\begin{Thm} \label{T:XYbound}
$$E\left[f_B^{(k)}(root) - f_A^{(k)}(root)\right] \leq \frac{\lambda\cdot e^{\lambda}}{k+1}.$$
\end{Thm}

\begin{proof}
We let $$A_k(x) = P(f_A^{(k)}(root)\geq x)$$ and $$B_k(x) = P(f_B^{(k)}(root)\geq x).$$ Clearly $A_k(x)$ and $B_k(x)$ are equal to 1 for $x<-\lambda/2$ and equal to 0 for $x>\lambda/2$. They are decreasing with a single discontinuity at $x=\lambda/2$ except for $A_0$, whose discontinuity is located at $x=-\lambda/2$. Pointwise we have $$A_0(x)\leq A_1(x) \leq A_2(x)\leq \dots \leq B_2(x) \leq B_1(x)\leq B_0(x).$$

Suppose in the following that $-\lambda/2 \leq x \leq \lambda/2$. Then $A_{k+1}(x)$ is the probability that there is no child $v_i$ of the root such that $l_i - f^{(k+1)}_A(v_i)<x$. In other words $A_{k+1}(x)$ is the probability that there is no event in the inhomogeneous Poisson process of $l_i$'s for which $f_A^{(k+1)}(v_i)> l_i - x$. Now notice that $f_A^{(k+1)}(v_i)$ has the same distribution as $f_B^{(k)}(root)$. Therefore $$A_{k+1}(x) = \exp\left(-\int_0^\infty B_k(l-x)\,dl\right) = \exp\left(-\int_{-x}^{\lambda/2}B_k(t)\,dt\right)$$ and similarly $$B_{k+1}(x) = \exp\left(-\int_{-x}^{\lambda/2} A_k(t)\,dt\right).$$
Differentiating, we see that $$A'_{k+1}(x) = -A_{k+1}(x)B_k(-x)$$ and $$B'_{k+1}(x) = -B_{k+1} (x)A_k(x).$$
We use the trick again (thereby promoting it to \emph{method}) and write this as
\begin{multline} \frac{d}{dx}\left(A_{k+1}(-x) + B_{k+1}(x)\right) = A_{k+1}(-x)B_k(x) - B_{k+1}(x)A_k(-x)\\= B_k(x)\cdot \left[A_{k+1}(-x)-A_k(-x)\right] + A_k(-x)\left[B_k(x)-B_{k+1}(x)\right],\end{multline} from which it follows that
 $$0\leq \frac{d}{dx}\left(A_{k+1}(-x) + B_{k+1}(x)\right) \leq \left[A_{k+1}(-x)-A_k(-x)\right] + \left[B_k(x)-B_{k+1}(x)\right].$$
By integrating over the interval $-\lambda/2\leq x\leq \lambda/2$, we find that \begin{multline} B_{k+1}(\lambda/2) - A_{k+1}(\lambda/2) = \int_{-\lambda/2}^{\lambda/2} \frac{d}{dx}\left(A_{k+1}(-x) + B_{k+1}(x)\right)\,dx \\ \leq  \int_{-\lambda/2}^{\lambda/2} \left(A_{k+1}(x) - A_k(x)\right)dx + \int_{-\lambda/2}^{\lambda/2} \left(B_k(x) - B_{k+1}(x)\right)dx.\end{multline}
Summing over $k$, we conclude that $$\sum_{k=0}^\infty \left(B_{k+1}(\lambda/2) - A_{k+1}(\lambda/2)\right) \leq \lambda.$$ Since $B_{k+1}(\lambda/2) - A_{k+1}(\lambda/2)$ is decreasing in $k$, it follows that \begin{equation} \label{G-F} B_{k+1}(\lambda/2) - A_{k+1}(\lambda/2) \leq \frac{\lambda}{k+1}.\end{equation}

Notice that $$A_{k+1}(\lambda/2) = \exp\left(-\lambda/2-E\left[f_B^{(k)}(root)\right]\right)$$ and $$B_{k+1}(\lambda/2) = \exp\left(-\lambda/2-E\left[f_A^{(k)}(root)\right]\right).$$
Since $A_{k+1}(\lambda/2) \geq \exp(-\lambda)$, we have \begin{multline} \exp\left(E\left[f_B^{(k)}(root)\right] - E\left[f_A^{(k)}(root)\right]\right) \\ = \frac{e^{-\lambda/2} \cdot e^{-E\left[f_A^{(k)}(root)\right]}}{e^{-\lambda/2}\cdot e^{-E\left[f_B^{(k)}(root)\right]}} = \frac{B_{k+1}(\lambda/2)}{A_{k+1}(\lambda/2)} \leq 1 + \frac{\lambda\cdot e^{\lambda}}{k+1}.\end{multline} Taking logarithms, we finally obtain $$E\left[f_B^{(k)}(root) - f_A^{(k)}(root)\right] \leq \log\left(1 + \frac{\lambda\cdot e^{\lambda}}{k+1}\right) \leq \frac{\lambda\cdot e^{\lambda}}{k+1}.$$
\end{proof}

\subsection{The limit as $k\to\infty$}
Since $E\left[f_B^{(k)}(root) - f_A^{(k)}(root)\right] \to 0$ as $k\to\infty$, $A_k$ and $B_k$ must converge to a common limit function that we denote by $F$. It turns out that $F$ can be determined explicitly.

\begin{Prop} \label{P:limitF}
On the interval $-\lambda/2 \leq x \leq \lambda/2$, the limit function $F$ is given by
\begin{equation} \label{explicitF} F(x) = \frac{1+q}{1+e^{(1+q)x}},\end{equation}
where $q = F(\lambda/2)$, and $q$ is determined by \begin{equation} \label{thetaq} \lambda = \frac{-2\log q}{1+q}.\end{equation}
\end{Prop}

Hence the common limit distribution of $f_A^{(k)}$ and $f_B^{(k)}$ can be regarded as a rescaled and truncated logistic distribution together with a point mass of $q$ at the point $\lambda/2$. If we let $\lambda\to\infty$, and consequently $q\to 0$, then this distribution converges to the logistic distribution, which is what we expect in view of the results in \cite{A01}.

\begin{proof} [Proof of Proposition \ref{P:limitF}]
On the interval $-\lambda/2\leq x\leq \lambda/2$ the limit function $F$ must satisfy \begin{equation} \label{Feq} F(x) = \exp\left(-\int_{-x}^{\lambda/2}F(t)\,dt\right),\end{equation} and hence $$F'(x) = -F(x)F(-x).$$
This means that $F'(x) = F'(-x)$, which in turn implies that $F(x) + F(-x)$ is constant. Putting $q = F(\lambda/2)$, we get \begin{equation} \label{constEq} F(-x) = 1 + q - F(x),\end{equation} and consequently $$F'(x) = -F(x)(1+q-F(x)).$$ Writing $$-\frac{F'(x)}{F(x)(1+q-F(x))} = 1$$ and integrating with respect to $x$, we obtain $$\log\left(\frac{1+q-F(x)}{F(x)}\right) = (1+q)x + C,$$ where putting $x=0$ reveals that $C=0$. Hence $$\frac{1+q-F(x)}{F(x)} = e^{(1+q)x},$$ from which we obtain \eqref{explicitF}.

Finally we would like to express $q$ in terms of $\lambda$, and either of the equations $F(-\lambda/2)=1$ or $F(\lambda/2) = q$ gives $$q = e^{-(1+q)\lambda/2},$$ which in turn yields \eqref{thetaq}.
\end{proof}

\subsection{The longest edge in a minimum partial matching}
The number $q=F(\lambda/2)$ is the probability that the starting point of the Exploration game is not included in the optimum diluted matching, and therefore it is the asymptotical density of unmatched vertices. The problem of finding the minimum matching that contains a specified number of edges is called the \emph{partial matching problem}. Equation \eqref{thetaq} gives an explicit relation between the density of a minimum partial matching and the length of its longest edge.

Suppose that $q$ is fixed. We study the minimum length partial matching that includes all but at most $qn$ vertices. The idea is that except for small fluctuations, we obtain such a matching by choosing $\lambda$ according to \eqref{thetaq}.

\begin{Prop} \label{P:maxCost}
Let $0<q<1$ and let $X_n$ be the length of the longest edge in the minimum partial matching on $K_n$ that includes all but at most $qn$ vertices. Then
$$ X_n \overset{\rm p}\to \frac{-2\log q}{1+q}.$$
\end{Prop}

\begin{proof}
For the moment think of $n$ as fixed. As the parameter $\lambda$ goes from 0 to infinity, the optimum diluted matching will pass through all the minimum partial matchings. Let $\lambda_q$ be the value for which the optimum diluted matching leaves out $qn$ vertices (rounded down). By \eqref{thetaq}, $$\lambda_q \overset{\rm p}\to \frac{-2\log q}{1+q}.$$

Clearly the optimum diluted matching contains no edge longer than $\lambda_q$. It remains to show that with high probability, the longest edge is not much shorter than $\lambda_q$. Take $\epsilon>0$. Then asymptotically almost surely, as $n\to\infty$, there will be some edge $e$ of length between $\lambda_q-\epsilon$ and $\lambda_q$ such that none of the vertices incident to $e$ has another edge shorter than $\lambda_q$. Such an edge must obviously be in the optimum $\lambda_q$-diluted matching.
\end{proof}

\section{The minimum diluted matching}
\subsection{Asymptotic total cost} \label{S:cost}
For fixed $\lambda$ we now want to find the asymptotical cost of the diluted matching problem as $n\to\infty$.
The cost (on average per vertex) of the penalties for vertices that are not included is concentrated at $q \lambda/2$. We therefore focus on the distribution of lengths of the edges in the optimum matching.
The final result (Proposition~\ref{P:finalCost} below) for the expected total length can be found with the methods of \cite{W10tsp}, but the distribution of the lengths of the participating edges is not available with that method.

\begin{Prop} \label{P:finalCost}
The expected total length of the edges in the optimum diluted matching is \begin{equation} \label{totalcost} \frac{n}{2}\cdot \int_q^1\frac{-2\log t}{1+t}\,dt + o(n),\end{equation} where as before $q$ is given by $$\lambda = \frac{-2\log q}{1+q}.$$
\end{Prop}

The upper limit of integration is thus given by the integrand being equal to $\lambda$. Since the cost of the penalties is $n \lambda q/2$, the total cost including penalties can be written as (dropping the error term) $$\frac{n}{2} \cdot\int_0^1 \min\left(\lambda, \frac{-2\log t}{1+t}\right)\,dt.$$

Although we started by fixing $\lambda$, it is apparently easier to express the final result in terms of the density $q$ of unmatched vertices.
It is of course no coincidence that the expression for $\lambda$ as a function of $q$ is the same as the integrand in \eqref{totalcost}, but we return to this point in a moment.

\begin{proof}
We derive \eqref{totalcost} with the method used in \cite{A01}, which in turn goes back to the physics literature. The edge lengths are uniform on $[0,n]$, and therefore the density function for the length of a particular edge is simply $1/n$ on that interval. The expected contribution to the total length of the optimum matching from an arbitrary edge $e$ between vertices $u$ and $v$ of $K_n$ is therefore \begin{equation} \label{contribution} \frac1n\cdot\int_0^{\lambda} z\cdot P(\text{participation given length $z$})\,dz.\end{equation}

The edge $e$ will participate in the optimum diluted matching if it is the optimal first move for Alice when the game starts at either of $u$ or $v$. We let $f(u)$ and $f(v)$ be the game-theoretical values of playing second if the game would start at $u$ or $v$ respectively, and be played with the edge $e$ deleted from the graph. If the game starts at $u$, Alice will go to $v$ in her first move if and only if the length $z$ of the edge $e$ satisfies $$z\leq f(u)+f(v)$$ (for a detailed argument see \cite{W09}).

The $(k, \lambda)$-neighborhoods of $u$ and $v$ (with $e$ deleted) can be approximated by two independent PWITs. Therefore the edge $e$ will participate in the optimum solution essentially if $z\leq f_1+f_2$, where $f_1$ and $f_2$ are independent and drawn from the distribution given by $F$ in Proposition~\ref{P:limitF}.

Hence apart from the scaling factor $1/n$, \eqref{contribution} is equal to $$\int_0^{\lambda} z\cdot P(z\leq f_1 + f_2)\,dz = \int_0^{\infty} z\cdot P(z\leq f_1 + f_2)\,dz.$$
Without using any particular properties of the probability distribution, we can rewrite this as \begin{multline} \label{rewrite} \int_0^\infty z \int_{-\infty}^\infty (-F'(x))\cdot P(f_2\geq z-x)\,dxdz \\= \int_0^\infty z \int_{-\infty}^\infty (-F'(x))F(z-x)\,dxdz.\end{multline}
With $u=z-x$, this becomes \begin{multline} \notag \int_{-\infty}^\infty F(u) \int_0^\infty z(-F'(z-u))\,dzdu \\= \int_{-\infty}^\infty F(u) \int_{-u}^\infty (x+u)(-F'(x))\,dxdu,\end{multline} and by partial integration, this is \begin{equation} \label{theIntegral} \int_{-\infty}^\infty F(u) \int_{-u}^\infty F(x)\,dxdu.\end{equation}

We can compute \eqref{theIntegral} using our explicit knowledge of the function $F$. But there is another method which is simpler and applicable to a wider range of problems. We introduce the function $$G(u) = \int_{-u}^\infty F(x)\,dx.$$ Clearly $G'(-u) = F(u)$, which means that \eqref{theIntegral} is transformed to
\begin{equation} \label{newIntegral} \int_{-\infty}^{\infty}G'(-u)G(u)\,du =
\int_{u=-\infty}^{u=\infty}G(u)\,dG(-u).\end{equation}
Now we begin to see some similarities to the calculations in Section~\ref{S:cavity}.
A simple interpretation of \eqref{newIntegral} is that it is the area under the curve (in the positive quadrant) when $G(u)$ and $G(-u)$ are plotted against each other. In order to find the value of this integral, we therefore only need to know the relation between $G(u)$ and $G(-u)$. By \eqref{Feq} and \eqref{constEq} we have $F(u) = e^{-G(u)}$ and $F(u) + F(-u) = 1+q$, which means that $$e^{-G(u)} + e^{-G(-u)} = 1+q.$$ Hence \eqref{newIntegral} is the area under the curve $$e^{-x} + e^{-y} = 1 + q,$$ in the $xy$-plane, which can also be expressed as \begin{equation} \label{qint}\int_0^{-\log q}{-\log\left(1+q-e^{-x}\right)}\,dx,\end{equation} where the upper limit of integration is obtained by putting $y=0$. This is precisely what can be obtained with the methods of \cite{W10tsp}.

A convenient way of handling \eqref{qint} is to differentiate with respect to $q$. The derivative is, after some simplification, $$\frac{2\log q}{1+q}.$$ By integrating back, \eqref{qint} is equal to $$\int_q^1 \frac{-2\log t}{1+t}\,dt,$$  which establishes \eqref{totalcost}.
\end{proof}

We already mentioned the observation that the integrand in \eqref{totalcost} is precisely the expression for $\lambda$ in terms of $q$. This gives a qualitative insight: The increase in total length of the minimum partial matching, if we require one more edge, is roughly equal to the length of the longest edge in the solution.
If the difference between the minimum partial matchings of $r$ and $r+1$ edges respectively is that the $(r+1)$-matching contains an edge between two vertices that do not participate in the $r$-matching, then of course that edge is the longest in the $(r+1)$-matching, and the difference in total length between the two matchings is exactly equal to the length of that edge. If the $(r+1)$-matching is obtained by replacing several edges in the $r$-matching, then the difference in total length of the two matchings can be larger than the length of the longest edge in the $(r+1)$-matching, but we can conclude here that, roughly speaking, for most $r$ it is not much larger.  We suspect that there is a shortcut from Proposition~\ref{P:maxCost} to Proposition~\ref{P:finalCost} based on this observation.

\subsection{The $h$-function}
Although it is not necessary for the computation of the expected cost of the optimum matching, it is interesting to find the conditional probability that an edge of given length participates in the solution. We therefore wish to determine the function $$h(x) = P(x\leq f_1 + f_2),$$ where $f_1$ and $f_2$ are independent and drawn from the limit distribution given by Proposition~\ref{P:limitF}, and $0 \leq x \leq \lambda$. We have \begin{multline} h(x) = q + \int_{-\lambda/2}^{\lambda/2} (-F'(u))F(x-u)\,du \\= q + (1+q)^3\int_{-\lambda/2}^{\lambda/2}\frac{1}{(1+e^{(1+q)u})(1+e^{-(1+q)u})}\cdot\frac{1}{1+e^{(1+q)(x-u)}}\,du. \end{multline} Here the term $q$ comes from the case that $f_1 = \lambda/2$, and the integral represents the case $f_1<\lambda/2$, when the density of $f_1$ at $u$ is $-F'(u)$.
With the substitution $t = e^{(1+q)u}$ we get $$du = \frac{dt}{(1+q)t}.$$ Moreover, since $e^{-(1+q)\lambda/2} = q$, the limits of integration $u=-\lambda/2$ and $u=-\lambda/2$ are equivalent to $t=q$ and $t=1/q$. Hence \begin{multline} \label{hfunction} h(x) = q + (1+q)^2 \int_{u = -\lambda/2}^{u = \lambda/2} \frac{t}{(1+t)^2(t+e^{(1+q)x})}\,dt \\= q + (1+q)^2 \int_{q}^{1/q} \frac{t}{(1+t)^2(t+e^{(1+q)x})}\,dt.\end{multline}
It can be verified that, writing $\alpha = e^{(1+q)x}$, the integrand has the primitive $$\frac{\alpha}{(\alpha-1)^2}\log\left(\frac{t+1}{t+\alpha}\right) + \frac{1}{(\alpha-1)(t+1)},$$ and after some simplification we get \begin{equation} \label{h} h(x) = q + \frac{(1+q)^3\alpha}{(\alpha-1)^2}\log\left(\frac{\alpha+q}{1+\alpha q}\right) - \frac{(1+q)^2(1-q)}{\alpha-1}. \end{equation}
If we put $q=0$, then $\alpha = e^x$, and we get $$h(x) = \frac{1-e^x+xe^x}{(e^x-1)^2},$$ which agrees with Theorem~2 of \cite{A01}.

Another special case is if we put $x=0$, in other words $\alpha=1$. This gives the probability $h(0)$ that an edge of zero length will participate in the minimum matching of density $1-q$. With $\alpha=1$, \eqref{h} does not make sense, but by evaluating \eqref{hfunction} we find that $$h(0) = \frac12 + q -\frac12q^2.$$

\section{The TSP}
\subsection{Relaxation and comply-constrain game}
As was described in \cite{W09}, the TSP is related to a ``refusal'' or ``comply-constrain'' version of Exploration: Whenever Alice is about to make a move, Bob has the right to forbid one of her move options, and vice versa. As before, a player can quit the game at cost $\lambda/2$.

The finite-$\lambda$ relaxation of the TSP is obtained by allowing any set of edges for which each vertex has degree at most 2, and where a penalty of $\lambda/2$ is paid for each missing edge at each vertex. Hence a vertex of degree 1 means a penalty of $\lambda/2$, while a vertex of degree 0 leads to a penalty of $\lambda$. In the case of the TSP the parity of the number $n$ of vertices is not an issue, and therefore equivalently the total penalty is $$\lambda\cdot (n-\#\text{ edges in the solution}).$$

The comply-constrain game leads to a different concept of valuation. Instead of \eqref{valdef}, we require
\begin{equation} \label{valdefTSP}
f(v) = \min(\lambda/2, {\rm min}_2(l_i-f(v_i))).
\end{equation}

Here $\min_2$ means second-smallest. We remark that equations equivalent to \eqref{valdefTSP} were derived in \cite{KM89, MP86a, MP86b} and also in \cite{A01}. We similarly redefine the partial valuations $f^{(k)}_A$ and $f^{(k)}_B$ in the obvious way. Again the crucial point is to prove that the expectation of $f_B^{(k)}(root) - f_A^{(k)}(root)$ tends to zero for large $k$. This time we obtain a slightly stronger bound:

\begin{Prop}
$$E\left[f_B^{(k)}(root) - f_B^{(k)}(root)\right] \leq \frac{e^\lambda}{k+1}.$$
\end{Prop}

\begin{proof}
We again let $$A_k(x) = P(f_A^{(k)}(root)\geq x)$$ and $$B_k(x) = P(f_B^{(k)}(root)\geq x).$$

For $-\lambda/2 \leq x \leq \lambda/2$, $A_{k+1}(x)$ is now the probability that there is at most one child $v_i$ of the root such that $l_i - f^{(k+1)}_A(v_i)<x$, that is, $A_{k+1}(x)$ is the probability that there is at most one event in the inhomogeneous Poisson process of $l_i$'s for which $f_A^{(k+1)}(v_i)> l_i - x$. Since again $f_A^{(k+1)}(v_i)$ has the same distribution as $f_B^{(k)}(root)$, we get
\begin{multline} \label{recurrence1}
A_{k+1}(x) = \left(1+\int_0^\infty B_k(l-x)\,dl\right)\exp\left(-\int_0^\infty B_k(l-x)\,dl\right) \\= \left(1+\int_{-x}^{\lambda/2}B_k(t)\,dt\right)\exp\left(-\int_{-x}^{\lambda/2}B_k(t)\,dt\right)\end{multline} and similarly
\begin{equation} \label{recurrence2}
B_{k+1}(x) = \left(1+\int_{-x}^{\lambda/2} A_k(t)\,dt\right)\exp\left(-\int_{-x}^{\lambda/2} A_k(t)\,dt\right).
\end{equation}

It is convenient to introduce the functions $$a_k(x) = \int_{-x}^{\lambda/2} A_k(t)\,dt$$ and $$b_k(x) = \int_{-x}^{\lambda/2} B_k(t)\,dt.$$
%Differentiating, we see that $$A'_{k+1}(x) = -A_{k+1}(x)B_k(-x)$$ and $$B'_{k+1}(x) = -B_{k+1} (x)A_k(x).$$
Using the trick again, we consider the quantity $$\Delta_k(x) = \frac{d}{dx}\left[(2+a_k(x))e^{-a_k(x)} + (2+b_k(-x))e^{-b_k(-x)}\right].$$
It is easily verified that $$\Delta_k(x) = A_k(-x)(B_k(x) - B_{k+1}(x)) + B_k(x)(A_{k+1}(-x) - A_k(-x)).$$
Since the function $(1+x)e^{-x}$ is decreasing in $x$, it follows inductively from the recurrence equations \eqref{recurrence1} and \eqref{recurrence2} that pointwise, $$A_0(x) \leq A_1(x) \leq A_2(x) \leq \dots \leq B_2(x) \leq B_1(x) \leq B_0(x).$$ Therefore $$0\leq \Delta_k(x) \leq \left[B_k(x) - B_{k+1}(x)\right] + \left[A_{k+1}(x) - A_k(x)\right].$$ Summing over $k$, we conclude that $$\sum_{k=0}^\infty \int_{-\lambda/2}^{\lambda/2} \Delta_k(x)\,dx \leq \lambda.$$
By the boundary conditions $a_k(-\lambda/2) = b_k(-\lambda/2) = 0$, this implies that $$\sum_{k=0}^\infty \left[(2+a_k(\lambda/2))e^{-a_k(\lambda/2)} - (2+b_k(\lambda/2))e^{-b_k(\lambda/2)}\right] \leq \lambda.$$

Observe that $$E[f_B^{(k)}(root) - f_A^{(k)}(root)] = b_k(\lambda/2) - a_k(\lambda/2).$$ Since $(2+x)e^{-x}$ is monotone decreasing for $x\geq 0$, it follows that $$(2+a_k(\lambda/2))e^{-a_k(\lambda/2)} - (2+b_k(\lambda/2))e^{-b_k(\lambda/2)} \leq \frac{\lambda}{k+1}.$$ Since the absolute value of the derivative of $(2+x)e^{-x}$ is $(1+x)e^{-x}$, which is decreasing, and trivially $b_k(\lambda/2) \leq \lambda$, it follows that $$b_k(\lambda/2) - a_k(\lambda/2) \leq \frac{\lambda e^\lambda}{(\lambda+1)(k+1)} \leq \frac{e^\lambda}{k+1}.$$ This completes the proof.
\end{proof}

\subsection{The finite-$\lambda$ integral equation}
The functions $A_k$ and $B_k$ thus converge to a common limit that we again denote by $F$, and which satisfies $$F(x) = \left(1 + \int_{-x}^{\lambda/2} F(t)\,dt\right) \exp\left( -\int_{-x}^{\lambda/2} F(t)\,dt\right).$$
If we write $G$ for the common limit of $a_k$ and $b_k$, that is, $$G(x) = \int_{-x}^{\lambda/2} F(t)\,dt,$$ then we obtain  $$G'(x) = (1+G(-x))e^{-G(-x)}.$$

This equation looks like the Krauth-M\'ezard-Parisi equation, but the difference is that for finite $\lambda$, we only require it to hold in the interval $[-\lambda/2, \lambda/2]$. Moreover, the boundary conditions depend on $\lambda$.
By the trick, again, $$(2+G(x))e^{-G(x)} + (2+G(-x))e^{-G(-x)} = C,$$ for some constant $C$ in the interval $2 < C < 4$.

If $C$ is fixed, the solution is unique: Supposing that we know $C$, $G(0)$ is determined by $$(2+G(0))e^{-G(0)} = C/2.$$ Let $\Lambda$ be the function mapping $G(x)$ to $G(-x)$, and as in Section~\ref{S:cavity} let $T(g)=(1+g)e^{-g}$. Then $$G'(x) T(G(-x)) = T(\Lambda(G(x))),$$ and again
\begin{equation} x = \int_{G(0)}^{G(x)} \frac{dt}{T(\Lambda(t))}.\end{equation}

This shows that $G$ is uniquely determined by $C$ (and vice versa), and in view of the results of the previous section, $C$ is therefore determined by $\lambda$.

\subsection{The length of the minimum tour}
In analogy with minimum matching, we obtain the cost of the minimum diluted tour as $$\int_0^\lambda zP(z\leq f_1+f_2)\,dz.$$ Through calculations analogous to those of Section~\ref{S:cost}, this can be transformed to the area under the curve given by $$(2+x)e^{-x}+(2+y)e^{-y} = 2-q,$$ where $2-q$ is the average degree of a vertex in the solution. Again this agrees with what is obtained in \cite{W10tsp}. The asymptotical total length of the optimum tour is then recovered in the limit $\lambda\to\infty$, leading to the calculations of Section~\ref{S:cavity}. The fact that $\lambda\to \infty$ corresponds to the TSP is again a nontrivial result and depends on the theorem of Frieze \cite{F04}.

\end{document}